\documentclass[12pt]{article}
\usepackage[pctex32]{graphics}

\usepackage{amssymb,latexsym,amsmath,epsfig,amsthm}
\newtheorem{theorem}{Theorem}

\newtheorem{example}{Example}

\begin{document}
	
	\title{Invariant measures for piecewise fractional linear maps}
	\author{Fritz Schweiger}
	
	\date{}
	\maketitle
	
	\vskip 30pt
	{\bf Abstract}
	\noindent 
	The first part deals with piecewise fractional linear maps with three branches. Given a map $T$  a map $S$ is called  a related map if some branches of $T$ are replaced by a 'flipped' branch, namely a branch of $1-T$. The main question is if $T$ and $S$ have a common invariant measure. The short second part presents invariant measures of a new type.

	
	\section*{0. Introduction}
	In this paper Moebius maps (= piecewise fractional linear maps) with $N$ or infinitely many branches are discussed (\cite{Sch95}, \cite{Sch06}, \cite{Sch16} \cite{Sch17}, \cite{Sch18a}, \cite{Sch18b}). There is a given partition $p_j$, $ j \in J$ of the unit interval $B = [0,1]$, where $J$ is a finite or countably infinite set such that the map $T: [0,1]  \to [0,1]$ is bijective from $]p_{j-1}, p_j[$, $j \in J$, onto $ ]0,1[$. The inverse map $V_j$  is called an {\em (inverse) branch} and is given as $$V_j (x)= \frac{ c_j + d_j x}{a_j + b_j x}, \, j \in J. $$
	The Jacobians of these maps are denoted as $$\omega_j = \omega_j(x) = \frac{\left| a_j d_j - b_j c_j\right| }{(a_j + b_l x)^2}, \, j \in J .$$  We will use the same notation for the associated matrices $$V_j =  \left( \begin{array}{cc}a_j &  b_j
		\\ c_j & d_j \end{array} \right), \, j \in J. $$
	Let $\epsilon = +$ stand for an increasing map and $\epsilon = -$ for a decreasing map. We call the map $T$ of type $(\epsilon_1, \epsilon_2, ...)$, $\epsilon_j \in \{+, -\}$, if $V_j$ is increasing or decreasing. If the parameters satisfy some conditions the map $T$ is ergodic and admits a $\sigma$-finite invariant measure. These conditions are different for increasing and decreasing branches and will be given later. \\
	In several investigations the question has been to find an explicit shape of the density of the invariant measure. In the author's papers three possibilities were used. A map $T^{*}:  B^{*} \to B^{*}$ is called a {\em dual map}  if $B^{*}$ is an interval and $T^{*}$ is a piecewise fractional linear map whose branches are given by the adjoint  matrices  
	$$V_j^{*}  =  \left( \begin{array}{cc}a_j &  c_j
		\\ b_j & d_j \end{array} \right), \, j \in J. $$
	Then the function
	$$h(x) = \int_{B^{*}} \frac{dy}{(1+xy)^2}$$
	is (up to a multiplicative factor) the density of the invariant measure for $T$. 
	If there is a map $\psi(t) = \frac{B +Dt}{A+Bt}$ such that $ \psi \circ T = T^{*} \circ  \psi$  (or equivalently $ \psi \circ V_k  = V_k^{*} \circ  \psi$) then the dual is called a {\em natural dual}. In the other case it is called an {\em exceptional dual}. 
	If  ${B^{*}} = [\eta, \theta]$, $\eta < \theta$, then the invariant density is in both cases 
	$$ h(x)= \frac{\theta}{1 + \theta x} - \frac{\eta}{1 + \eta x}.$$
	If the set ${B^{*}}$ shrinks to a point $\xi$ one can see this case as a natural dual as well as an exceptional one. We call it a {\em singular dual}. In this case the density of the invariant measure is
	$h(x)= \frac{1}{(1+\xi x)^2}$.\\
	However, there are other examples where we can write down an explicit form of the density of the invariant measure. This method has been called a {\em $1$-step extension} (see \cite{Sch24}). An example is the following map.
	Let 
	$$Tx = \frac{x}{1-2x}, \, \,0 \leq x \leq \frac{1}{3} $$ $$ Tx =  \frac{1}{x} - 2, \, \, \frac{1}{3} < x \leq \frac{1}{2} $$
	$$ Tx = \frac{1}{x} - 1,  \, \,\frac{1}{2} < x \leq 1 $$
	
	The density of the invariant measure is given by
	$$ h(x) = \frac{1}{x} \sum_{j=0}^\infty \left(  \frac{1}{1+2jx}-\frac{1}{1+(2j+1)x} \right). $$
	This is a function with infinitely many poles, and hence not rational.
	The set $B^{*}= \bigcup_{j=0}^{[\infty}[2j, 2j+1]$ is  associated with this density.\\
	The first part of the paper deals with maps with three branches. One or more branches of the map $T$ are {\em 'flipped'}, this means, they are  substituted by the branch of $1 -T$. This map will be called $S$. 	
	We investigate the question if the map $T$ and a related map $S$ have a common invariant measure.\\
	The second part deals with the discovery of new explicitly given densities for some special cases. Example 1 in Section 4 presents a map with four branches. The density of its invariant measure is
	$$h(x)= \frac{1}{1+x} - \frac{1}{2+x} +\frac{1}{3-x}.$$

\section*{1. Related maps with one 'flip'}
	Here we consider maps $S$ where one branch is substituted by the branch related to $1-T$. Note that in this case the entries $A$, $B$, and $D$ of a possible map $\psi$ is determined by the two unchanged branches.  We first look on the map $S_1$ where the branch $V_\lambda$ is replaced by $$W_\lambda = \left( \begin{array}{cc} 3 \lambda &  -3 \lambda +3
		\\ \lambda & - \lambda \end{array} \right) .$$ Then we consider the map $S_2$ where the branch $V_\mu$ is replaced by $$W_\mu = \left( \begin{array}{cc} 3 \mu &  -3\mu +9
		\\ 2 \mu & - 2 \mu +3 \end{array} \right),$$ and the map $S_3$ where the branch $V_\nu$ is replaced by $$W_\nu = \left( \begin{array}{cc} \nu  &  -\nu +3
		\\ \nu & -\nu +2 \end{array} \right).$$
	\begin{theorem}
		(1) The maps $T$ and $S_1$ both have a natural dual and the same invariant measure if the conditions 
		$$(CT) \, \, \, \lambda \mu + \mu = \lambda \nu + 3 \lambda$$ and
		$$(CS_1) \, \, \, \lambda \mu \nu + 3 \lambda \mu + 12 \mu = 9 \nu +27$$
	are satisfied.	\\
		(2) The maps $T$ and $S_2$ both have a natural dual and the same invariant measure if the conditions 
	$$(CT) \, \, \, \lambda \mu + \mu = \lambda \nu + 3 \lambda$$ and
	$$(CS_2) \, \, \, \lambda \mu \nu = 9 $$
	are satisfied.	\\	
		(3) The maps $T$ and $S_3$ both have a natural dual and the same invariant measure if the conditions 
	$$(CT) \, \, \, \lambda \mu + \mu = \lambda \nu + 3 \lambda$$ and
	$$(CS_3) \, \, \, 4 \lambda \nu - \mu \nu - \lambda \mu \nu + 3 \lambda +3 = 0$$
	are satisfied.			
	\end{theorem} 
	\begin{proof}
		(1) We start with  $A = 4 \mu - 3 \nu -9 $, $B = -6\mu + 3 \nu +9$, and   $D = - \mu \nu + 9 \mu - 3\nu -9 $. The condition $A(3 -3\lambda) - 4\lambda B= \lambda D $ gives $(CS_1)$.\\ 
		We insert $\nu = \frac{\lambda \mu + \mu - 3 \lambda}{\lambda}$ into $(CS_1)$ and obtain $$\lambda^2 \mu +\lambda(\mu +3)-9 =0 .$$ $\lambda = \frac{3}{4} $, $\mu = \frac{36}{7}$ , and $\nu =9$ are solutions.\\
		(2) We take $A = 3-\lambda$, $B= 3\lambda -3$, and $D=3 +\lambda \nu -6 \lambda$. The condition
		$A(-3 \mu +9) + B(-5\mu +3)= 2\mu D$ gives $(CS_2)$. \\ We insert $\lambda = \frac{9}{\mu \nu}$ into $(CT)$ and obtain $$\mu^2 \nu + 9\mu -9\nu -27 = 0 .$$ Then $\lambda = \mu =3$ and $\nu =1$ are solutions.\\
		(3) We now take $A = 3-\lambda$, $B= 3\lambda -3$, and $D=\lambda \mu -9 \lambda + \mu +3$. Then 
		$A(3- \nu) + B(2 - 2\nu ) = \nu D$ leads to $(CS_3)$.\\
		If we muliply $(CT)$ by $\nu$ and add it to $(CS_3)$, we obtain $$\lambda \nu^2 -\lambda \nu - 3 \lambda -3= 0.$$ Here, $\lambda = \frac{27}{13} $, $\mu = \frac{153}{40}$ , and $\nu =\frac{8}{3}$ are solutions.		
	\end{proof}

	\section*{2. Related maps with two 'flips'}
	Again, we have to look at three cases. $S_{12}$ has the branches $W_{\lambda}$ and $W_{\mu}$, $S_{23}$ the branches $W_{\mu}$ and $W_{\nu}$, and $S_{13}$ the branches $W_{\lambda}$ and $W_{\nu}$.
	\begin{theorem}
	(1) If the conditions 
	$$(CT) \, \, \, \lambda \mu + \mu = \lambda \nu + 3 \lambda$$ and
	$$(CS_{12}) \, \, \, \mu \nu - \lambda \nu - 3\lambda -3 =0$$ are satisfied then the maps $T$ and $S_{12}$ both have a natural dual but they have the same invariant measure only in the linear case. \\
	(2) If the conditions 
	$$(CT) \, \, \, \lambda \mu + \mu = \lambda \nu + 3 \lambda$$ and
	$$(CS_{23}) \, \, \, \lambda \mu \nu + 3\mu - 9 \nu +3 \lambda \mu  =0$$ are satisfied then the maps $T$ and $S_{12}$ both have a natural dual but they have the same invariant measure only in the linear case. \\
	(3) If the conditions 
	$$(CT) \, \, \, \lambda \mu  + \mu = \lambda \nu  + 3 \lambda$$ and
	$$(CS_{13}) \, \, \, \lambda \mu \nu + 3\mu \nu - 9 \nu -9 =0$$ are satisfied then the maps $T$ and $S_{13}$ both have a natural dual. They have the same invariant measure only in the linear case. 
	\end{theorem}
\begin{proof}
(1) Using the branches 	$W_{\mu}$ and $V_{\nu}$, we obtain $A'= \mu \nu -3$, $B' = 9 -\mu \nu$, and $D'= -18 + 3 \nu + \mu \nu$. The condition $A'(-3\lambda +3) - 4\lambda B' = \lambda D'$ gives condition $(CS_{12})$. If we insert $\nu = \frac{\lambda \mu + \mu - 3 \lambda}{\lambda}$ into this condition we obtain the quadratic equation 
$$\mu^2 (1+ \lambda) - \mu(4 \lambda + \lambda^2) - 3\lambda  =0.$$
Then  $\lambda = \frac{1}{2} $, $\mu = 2$ , and $\nu =3$ are solutions.\\
Clearly, $$\frac{B+D}{A+B} = \frac{\nu -3}{2}= \frac{B'+D'}{A'+B'},$$  but $$\frac{B}{A} = \frac{3\lambda -3}{3 -\lambda} =  \frac{B'}{A'} = \frac{9 - \mu \nu}{\mu \nu -3}$$ leads to $\lambda \mu \nu = 9$. We insert $\nu = \frac{9}{\lambda \mu }$ into equation $(CT)$ and obtain a second quadratic  equation
$$(1+ \lambda)\mu^2  -3 \lambda \mu - 9  =0.$$ Then we find $\mu = 3$. The first quadratic equation shows $\lambda =1$ and condition $(CT)$ gives $\nu =3$, the linear case.\\
(2) In this case we find $A = A' = 3 -\lambda$, $B=B' = 3 \lambda  -3$, and $D' = \frac{9-6\lambda \mu + 3 \mu }{\mu} = 
\frac{3 + 3 \lambda + 3 \nu -5 \lambda \nu}{\nu}$ and obtain in a similar way condition $(CS_{23})$. A solution for conditions $(CT)$ and $(CS_{23})$ is given by $\lambda =1$, $\mu = \frac{9}{2}$, and $\nu =6$.\\
From $$\frac{B+D}{A+B} = \frac{\nu -3}{2} = \frac{B+D'}{A+B}= \frac{9-3\lambda \mu}{2 \lambda \mu}$$ we deduce $D=D'$ and $\lambda \mu \nu =9$ .  Then we obtain the quadratic equation
$$(\lambda +1)\mu^2 - 3\lambda \mu - 9 = 0$$ and hence $\mu =3$.
If we insert $\nu = \frac{\lambda \mu + \mu - 3 \lambda}{\lambda}$ into $(CS_{23})$ then we obtain a second quadratic equation
$$ (\lambda^2 + \lambda)\mu^2 - (6 \lambda +9)\mu  +27 \lambda =0$$ and $\lambda =1$.
 Again we find the linear case.\\
(3) In this case we take $A' = 2 \lambda \nu + 2 \lambda$, $B' = -2 \lambda \nu + 3 \nu - 3 \lambda$, and $D' = 2 \lambda \nu + 6 \lambda - 6 \nu +6$. Then we obtain $(CS_{13})$.  A solution for conditions $(CT)$ and $(CS_{13})$ is given by $\lambda =2$, $\mu = 3$, and $\nu = \frac{3}{2}$.\\
 From  $$ \frac{\nu -3}{2}=\frac{B+D}{A+B}= \frac{B'+D'}{A' +B'} = \frac{3 \lambda - 3 \nu +6}{3 \nu - \lambda}$$ we find $\lambda = \frac{3\nu^2 - 3\nu -12}{\nu +3}$.  Since $$\frac{3 \lambda -3}{3- \lambda} = \frac{B}{A} = \frac{B'}{A'}= \frac{- 2\lambda \nu + 3 \nu - 3 \lambda}{2 \lambda \nu + 2 \lambda }$$ we deduce the equation
$$\nu^5 - \nu^4 -8 \nu^3 + 15\nu +9 = (\nu -3)(\nu^2 -3)(\nu +1)^2=0.$$ However, $\nu =3$ leads to the linear case and $\nu = \sqrt{3}$ gives $\lambda < 0$.\\
If $\frac{B}{A} = \frac{B' +D'}{A'+B'}$, namely $$\frac{3 \lambda -3}{3- \lambda}= \frac{3 \lambda - 3 \nu +6}{3 \nu - \lambda},$$ we obtain $\lambda \nu =3$. If we insert this value in $(CS_{13})$ then we we get $\mu + \mu \nu = 3 \nu +3$ and eventually $\mu = 3$. 
\end{proof}

\section*{3. Related maps with three 'flips'}
Here, $S_{123}$ is the only case. However, the results are a little bit surprising.
\begin{theorem}
The conditions 
		$$(CT) \, \, \, \lambda \mu + \mu = \lambda \nu + 3 \lambda$$ and 
		$$ (CS_{123}) \,\,\, \lambda \mu - 4 \lambda \nu + \mu \nu - 3 \lambda + 4 \mu - 3 \nu =0$$ characterize a natural dual of of $T$ and $S_{123}$. These maps have the same invariant measure if and only if $\mu =3$ and $\lambda \nu =3$.
\end{theorem}
\begin{proof}
We obtain $A' = \lambda + \lambda \mu$, $B'= - 3\lambda + 2 \mu - \lambda\mu$, and $D'= 9 \lambda - 5 \mu + \lambda \mu +3$. The condition
$$ (CS_{123}) \,\,\, \lambda \mu - 4 \lambda \nu + \mu \nu - 3 \lambda + 4 \mu - 3 \nu =0$$ follows in the usual way. Examples are easy to find, e. g. $\lambda = 2$ and $\mu = \nu =6$.\\
If $\mu =3$ and $\lambda \nu =3$, then a natural dual exists for both maps. We can take $A' = 2 \lambda$, $B' = 3- 3 \lambda$, and $D' = 6 \lambda -6$ and furthermore 
$A = 3 - \lambda$, $B = 3 \lambda - 3$, and $D = 6 -6 \lambda$. Then we calculate $\frac{B}{A} = \frac{B' + D'}{A'+B'}$ and $\frac{B+D}{A+B}=\frac{B'}{A'}$. A nice example is given by $\lambda  =3$,  $\mu = 3$, and $ \nu =1$.\\
If $\frac{B}{A} = \frac{B' + D'}{A'+B'}$ and $\frac{B+D}{A+B}=\frac{B'}{A'}$ then we deduce 
 $\mu(\lambda +1) = 3 (\lambda +1) $. This gives $\mu =3$. 
If $\frac{B}{A} = \frac{B'}{A'}$ and $\frac{B+D}{A+B}=\frac{B'+D'}{A'+B'}$ then we deduce two equations,namely
$$\lambda^2 \mu  + \lambda ( -\mu^2 + 4 \mu +3) - \mu^2 = 0$$
and 
$$\lambda^2 \mu  +  \lambda (\mu +3) -3 \mu = 0.$$ Hence $\lambda(\mu^2 - 3 \mu) = 3\mu - \mu^2. $ Since $\lambda \neq -1$, we get $\mu = 3$.
\end{proof}
 {\bf Remark}
 This paper covers only a small fraction of possible combinations  but it tries to show what can be expected. A last example will be given. We look at the map $T$ with an exceptional dual of type $\lambda \nu \mu$ or $\mu \nu \lambda$  as described in the introduction. We look for parameters $\lambda$, $\mu$, and $\nu$ such that the map $S_{123}$ has a natural dual wich implies condition 
 $ (CS_{123})$. Clearly, the linear case with $\lambda =1$, $\mu = 3$, and $\nu =3 $ is a solution but other solutions seem to unusually complicated. An example is given by $\lambda $, the positive root of $448 \lambda^2 +283 \lambda - 1113 =0$, $\mu = \frac{63}{16}$, and $\nu = \frac{7 \lambda + 7}{4}$. For the exceptional dual of $T$ one finds $\eta = \frac{1}{2}$ and $\xi = \frac{3 \lambda - 3}{3 - \lambda} = \frac{-257 +448 \lambda}{628}$.
 
  \section*{4. New measures through the back door}
 
 The starting point for this paper was a second look at a generalization of continued fractions (see \cite{Da} and \cite{Sch21}). The proof needs only the following  well known theorem.
 \begin{theorem}
 Let $U: [0. 1] \to [0,1]$ be a piecewise fractional map with invariant measure $\nu$ and $Z: [0. 1] \to [0,1]$  be a piecewise fractional map with invariant measure $\mu$.
 If  a piecewise fractional linear map $\phi$ with $\phi[0,1] = [0,1]$ satifies $\phi \circ U = Z \circ \phi$, then $\nu(E) = \mu( \phi^{-1}E)$ is an invariant measure for $U$.
 \end{theorem}
\begin{proof}
$$\nu(T^{-1}E ) = \mu(\phi^{-1}T^{-1}E)= \mu(S^{-1}\phi^{-1}E) = \mu(\phi^{-1}E)= \nu(E).$$
 \end{proof}
Let $\zeta_{k}$ denote the inverse branches of $\phi$. If $h$ is the density of $\mu$ and $g$ the density of $\nu$ then
$$g(x) = \sum_k h(\zeta_k(x)) \left| \zeta'_k (x) \right| .$$
If $Ux = T(Sx)$ and $Zx = S(Tx)$ then we take $\phi = S$. If $h$ is the density of the invariant measure for $U$ and $g$ the density of the invariant measure for $Z$ then
$$g(x)= \sum_j h(V_jx) \omega_j(x).$$
We now give some examples and remark that the densities given by this device have no obvious relation to a possible dual map of $Z^{*}$. There are some exceptions (see Remark and Example 6). Clearly, if $S$ and $T$ have the same invariant measure, this measure is also invariant for $U$ and $Z$.
\begin{example}
Let $S$ be the map on $[0,1]$ with the inverse branches $V_{\alpha}x = \frac{1-x}{2+x}$ and $V_{\beta}x = \frac{3-2x}{3-x}$ and $T$ be the map on $[0,1]$ with the inverse branches $V_{\gamma}x = \frac{3-3x}{6-x}$ and $V_{\delta}x = \frac{3-2x}{3-x}$. Then $U = T \circ S$ has the inverse branches $V_{\alpha \gamma }x = \frac{3+2x}{15 -5x}$, $V_{\alpha \delta }x = \frac{x}{9-4x}$, $V_{\beta \gamma}x = \frac{4+x}{5}$ , and $V_{\beta \delta}x = \frac{3+x}{6-x}$. This map has a natural dual on
$B^{*} = [-\frac{1}{2}, 0] $, Therefore the density of the invariant measure is $h(x)= \frac{1}{2-x}$. The density of the invariant measure of $Z=S \circ T$ is $$g(x)= h(V_{\alpha}x)\omega_{\alpha}(x) + h(V_{\beta}x)\omega_{\beta}(x) = \frac{1}{1+x} - \frac{1}{2+x} +\frac{1}{3-x}.$$
If one uses the inverse branches of $Z$, namely 
$V_{ \gamma \alpha}x = \frac{3+6x}{11 + 7x}$, $V_{\beta \gamma }x = \frac{3x}{15-4x}$, $V_{\delta \alpha}x = \frac{4+5x}{5+4x}$ , and $V_{\delta \beta}x = \frac{3+x}{6-x}$ one can also verify Kuzmin's equation for $g$ by direct calculation.
\end{example}
The next examples are given in shorter form!
\begin{example}
Let $V_{\alpha} x = \frac{x}{1+x}$ and  $V_{\beta} x = \frac{1}{2-x}$ be the inverse branches of $S$ and  $V_{\gamma} x = \frac{x}{3-x}$ and $V_{\delta} x = \frac{1}{2-x}$ be the inverse branches of $T$. The map $U = T \circ S $ has the density $h(x) = \frac{1}{1-x}$ and $Z= S \circ T$ the density $g(x)= \frac{1}{1+x} + \frac{1}{1-x} - \frac{1}{2-x}$.
\end{example}
\begin{example}
If $V_{\alpha} x = \frac{1-x}{2}$, $V_{\beta} x = \frac{1+x}{2}$, $V_{\gamma} x = \frac{x}{1+x}$,	 and $V_{\delta} x = \frac{1+x}{1+3x}$, then $U$ has the invariant density $h(x)= \frac{1}{x}$ and $Z$ has the invariant density $g(x)= \frac{1}{1-x} +  \frac{1}{1+x}$. This density looks familiar, but the branches of the dual map of $Z$ give no clear  explanation.  
\end{example}
\begin{example}
	Let $Sx = \frac{1}{x} -k$, $ k= \lfloor \frac{1}{x} \rfloor$ and $Tx = a x - j$ , $a \geq 2$ a fixed natural number and $j = \lfloor ax \rfloor$ Then 
	$$U x =T(Sx) = \frac{a}{x} -(ak + j).$$ 
	The density of its invariant measure is well known, namely $h(x)= \frac{1}{a+x}$. \\
	Now we look at $$Zx = S(Tx) = \frac{1+kj - akx}{-j +ax}.$$ Our method gives the density of the invariant measure as
	$$ g(x)= \sum_{m=1}^{\infty}h(\frac{1}{m+x})\frac{1}{(m+x)^2} = \sum_{m=1}^{\infty}\frac{1}{(am+1+ax)(m+x)}.$$
	\end{example}
\begin{example}
	If $Sx = \frac{x}{1-x}$, $k= \lfloor \frac{x}{1-x} \rfloor$ and $Tx = ax- \lfloor ax \rfloor $, then the density for $U =T\circ S$ is given by $h(x)= \frac{1}{a-1+x}$ and the density for $Z = S \circ T$ is 
	$$g(x)= \sum_{m=0}^{\infty} \frac{1}{(am + a-1 +ax)(m+1+x)}.$$
\end{example}
{\bf Remark} If all  branches of the map $U$ have a common fixed point $\kappa$ then $\xi= - \frac{1}{\kappa}$ is a common fixed point of the branches of $U^{*}$. Then $U$ has a singular dual with density $\frac{1}{(1+ \xi x)^2}$, and $Z$ has the singular dual with density $\frac{1}{(1+ \eta x)^2}$ where $ \eta = S \xi$. 
\begin{example}
Let $Sx = \frac{2-4x}{2+x}, 0 \leq x < \frac{1}{2}$ and $Sx = \frac{-3 +6x}{2+x}, \frac{1}{2} \leq x \leq 1.$ The second map $T$ is given as
$Tx = \frac{4x}{3-2x}, 0 \leq x < \frac{1}{2}$ and $Tx = \frac{4 -4x}{1+2x} , \frac{1}{2} \leq x \leq 1.$ The map $U = T \circ S $ has the density $h(x) = \frac{1}{(2+x)^2}$ and $Z= S \circ T$ the density $g(x)= 1$ which corresponds to $\eta = S \frac{1}{2} = 0$.
\end{example}
 {\bf Remark} This method was used to give a derivation for the density of the invariant measure of an $f$-expansion on $[0,1]$, namely the map $f^{-1}x = Ax + (2-A)x^2$ mod ${1}$ (see \cite {Th}) using the invariant measure of $Tx = Ax + (4-2A)x^2, 0 \leq x < \frac{1}{2} \, , Tx = 1-A + (-4+3A)x + (4-2A)x^2, \frac{1}{2} \leq x \leq 1$ (see \cite {Sch18a}).

Fritz Schweiger\\
University of Salzburg, FB Mathematik\\
Hellbrunnerstr. 34, A 5020 Salzburg, Austria\\
fritz.schweiger@plus.ac.at
	\end{document}